\documentclass[11pt, reqno]{amsart}
\usepackage{amsmath}%
\usepackage{amssymb}%
\usepackage[english]{babel}
\usepackage[utf8]{inputenc}
\usepackage{hyperref}
\usepackage{amsfonts}%
\usepackage{setspace}
\usepackage{enumerate}
\usepackage{graphicx}
\usepackage{xcolor}
\usepackage[margin=1.25in]{geometry}
\usepackage{mathabx}
\usepackage[T1]{fontenc}

\usepackage{appendix}
\usepackage[normalem]{ulem}

    \numberwithin{equation}{section}

\newcommand{\F}{\mathbb{F}}

\newcommand{\Hq}{\mathbb{H}}

\newcommand{\N}{\mathbb{N}}
\newcommand{\C}{\mathbb{C}}

\newcommand{\D}{\mathbb{D}}
\newcommand{\bR}{\mathbb{R}}

\newcommand{\B}{\mathcal{B}}
\newcommand{\Ca}{\mathcal{C}}
\newcommand{\Fed}{\mathcal{F}}

\newcommand{\Hi}{\mathcal{H}}
\newcommand{\K}{\mathcal{K}}
\newcommand{\Ls}{\mathcal{L}}

\newcommand{\ich}{{\rm iconv\,}}

\newcommand{\im}{{\rm Im\,}}
\newcommand{\re}{{\rm Re\,}}

\newcommand{\w}[1]{\omega^{(#1)}}

\newcommand{\inn}[2]{\langle #1, #2 \rangle}
\newcommand{\innt}[1]{\langle T#1, #1 \rangle}
\newcommand{\e}{\varepsilon}
\newcommand{\ch}{\text{conv}}

\newcommand{\Span}{\text{span}\,}
\newcommand{\diag}{\text{diag}\,}

\newcommand\restr[2]{{
  \left.\kern-\nulldelimiterspace 
  #1 
  \right|_{#2} 
  }}

\newtheorem{theorem}{Theorem}[section]

 \newtheorem{corollary}[theorem]{Corollary}
 
 \newtheorem{lemma}[theorem]{Lemma}
 \newtheorem{proposition}[theorem]{Proposition}
 \theoremstyle{definition}
 \newtheorem{definition}[theorem]{Definition}
 \newtheorem{remark}[theorem]{Remark}
 
 \numberwithin{equation}{section}

\makeatletter
\newtheorem*{rep@theorem}{\rep@title}
\newcommand{\newreptheorem}[2]{
\newenvironment{rep#1}[1]{%
\def\rep@title{#2 \ref{##1}}%
\begin{rep@theorem}}%
{\end{rep@theorem}}}
\makeatother

\newreptheorem{corollary}{Corollary}

\begin{document}


\title{On the convexity of the quaternionic essential numerical range}

\author[L. Carvalho]{Lu\'{\i}s Carvalho}
\address{Lu\'{\i}s Carvalho, ISCTE - Lisbon University Institute\\    Av. das For\c{c}as Armadas\\     1649-026, Lisbon\\   Portugal}
\email{luis.carvalho@iscte-iul.pt}
\author[Cristina Diogo]{Cristina Diogo}
\address{Cristina Diogo, ISCTE - Lisbon University Institute\\    Av. das For\c{c}as Armadas\\     1649-026, Lisbon\\   Portugal\\ and \\ Center for Mathematical Analysis, Geometry,
and Dynamical Systems\\ Mathematics Department,\\
Instituto Superior T\'ecnico, Universidade de Lisboa\\  Av. Rovisco Pais, 1049-001 Lisboa,  Portugal
}
\email{cristina.diogo@iscte-iul.pt}
\author[S. Mendes]{S\'{e}rgio Mendes}
\address{S\'{e}rgio Mendes, ISCTE - Lisbon University Institute\\    Av. das For\c{c}as Armadas\\     1649-026, Lisbon\\   Portugal\\ and Centro de Matem\'{a}tica e Aplica\c{c}\~{o}es \\ Universidade da Beira Interior \\ Rua Marqu\^{e}s d'\'{A}vila e Bolama \\ 6201-001, Covilh\~{a}}
\email{sergio.mendes@iscte-iul.pt}
\author[H. Soares]{Helena Soares}
\address{Helena Soares, ISCTE - Lisbon University Institute\\    Av. das For\c{c}as Armadas\\     1649-026, Lisbon\\   Portugal}
\email{helena.soares@iscte-iul.pt}
\subjclass[2010]{47A12, 47S05}

\keywords{quaternions, numerical range, essential numerical range}
\date{\today}

\maketitle

\begin{abstract}

The numerical range in the quaternionic setting is, in general, a non convex subset of the quaternions. The essential numerical range is a refinement of the numerical range that only keeps the elements that have, in a certain sense, infinite multiplicity. We prove that the essential numerical range of a bounded linear operator on a quaternionic Hilbert space is convex. A quaternionic analogue of Lancaster theorem, relating the closure of the numerical range and its essential numerical range, is also provided.
%
\end{abstract}
\maketitle
\maketitle
\section*{Introduction}

Let $\mathbb{F}$ be the field of complex numbers or the skew field $\Hq$ of Hamilton quaternions. Let $\mathcal{H}$ be a Hilbert space over $\mathbb{F}$ and let $T$ be a bounded linear operator on $\mathcal{H}$. The numerical range of $T$ is the set
\[
W(T)=W_{\mathbb{F}}(T)=\{\langle Tx,x\rangle:\|x\|=1, x\in\mathcal{H}\},
\]
where $\langle \cdot\,,\cdot\,\rangle:\mathcal{H}\times\mathcal{H}\to\mathbb{F}$ is the inner product on $\mathcal{H}$. This subset of $\mathbb{F}$ was introduced and studied by Toeplitz in 1918, who proved that, when $\F=\C$, the outer boundary of $W(T)$ is a convex curve and conjectured that the whole numerical range was convex, see \cite{To}. Shortly after, in 1919, Hausdorff \cite{Ha} proved the conjecture. Since then, this result is known as the Toeplitz-Hausdorff Theorem.

Over the years, the investigation of the numerical range continuously increased, including the cases of linear operators on infinite dimensional complex Hilbert spaces and complex Banach spaces. In 1951, Kippenhahn \cite{Ki} introduced the study of numerical range for quaternionic operators, i.e, when $\F=\Hq$. Soon it became evident that, although sharing many properties of its complex counterpart, the quaternionic numerical range was no longer always convex. 
The bild of an operator $T$, also introduced in \cite{Ki}, is the intersection $B(T)=W_{\Hq}(T)\cap\C$. Since every quaternion is, up to unitary equivalence, a complex number, many properties of the numerical range are encoded in the bild, including convexity. In fact, $W_{\Hq}(T)$ is convex if, and only if, $B(T)$ is convex, see \cite{CDM3}. However, the upper bild $B^+(T)$, which is the intersection of $W_{\Hq}(T)$ with the closure of the upper half-plane, is always convex. 
The pursuit of convexity remained an important issue in the quaternionic setting, with Au-Yeung establishing in \cite{Ye1} necessary and sufficient conditions for $W_{\mathbb{H}}(T)$ to be convex.

In a series of recent papers \cite{CDM1} - \cite{CDM5} the convexity and shape of the numerical range of quaternionic matrices have been studied by the first three named autors. The notion of S-spectrum in \cite{CGSS} and its relation with the numerical range on infinite dimensional quaternionic Hilbert spaces was adressed in the recent preprint \cite{CDM6}. Another geometric object in the realm of infinite dimensional Hilbert spaces is the essential numerical range of an operator $T$. It is defined as the set
\[
W_{e}(T)=W_{e,\F}(T)=\bigcap_{K\in\K(\Hi)}\overline{W(T+K)},
\]
where $\mathcal{K}(\mathcal{H})$ denotes the set of compact operators on the $\F$-Hilbert space $\mathcal{H}$. Taking $K$ to be the zero operator in the above definition, we see that $W_{e,\F}(T)\subseteq\overline{W_{\F}(T)}$.

This paper is devoted to the study of the essential numerical range in the quaternionic setting. The main result is theorem \ref{thm convexity} where we show that, for $\F=\Hq$, the essential numerical range $W_{e}(T)=W_{e,\Hq}(T)$ is always a convex set. Thus, at least convexity of this essential part of the numerical range is guaranteed even in the quaternionic setting. We emphasize that this is a surprising and unexpected result since the essential numerical range is the intersection of non-convex sets and nothing indicates it is convex in its formulation.

To secure this result we use a general property (lemma \ref{lemmasequence}): given a pair of unitary sequences $x_n^{(1)}, x_n^{(2)}$ and $T\in \B(\Hi)$, a judiscious choice of $N$ and $M$ shows that the following vectors are close to orthogonal
\[
\inn{x_N}{y_M} \approx  \inn{Tx_N}{y_M} \approx \inn{T^*x_N}{y_M}\approx 0.
\]
We can then form an essential sequence (see definition \ref{def_ess_seq}) for the convex combination $\alpha^2 \omega^{(1)}+\beta^2\omega^{(2)}$, $\omega^{(1)}, \omega^{2}\in W_e(T)$, with elements of the form $\alpha x_N^{(1)}+\beta y_M^{(2)}$. The referred quasi orthogonality implies that
\begin{align*}
&\|\alpha x_N^{(1)}+\beta x_M^{(2)}\|^2\approx \alpha^2\inn{x_N^{(1)}}{x_N^{(1)}}+\beta^2 \inn{x_M^{(2)}}{x_M^{(2)}}=1\\
&\langle T(\alpha x_N^{(1)}+\beta x_M^{(2)}), \alpha x_N^{(1)}+\beta x_M^{(2)}\rangle\approx \alpha^2\langle T x_N^{(1)}, x_N^{(1)}\rangle+\beta^2 \langle T x_M^{(2)}, x_M^{(2)}\rangle\approx \alpha^2 \omega^{(1)}+\beta^2\omega^{(2)}.
\end{align*}

%
%
%

We finish the paper with theorem \ref{thm lancaster quaternion}, where we prove a quaternionic version of Lancaster theorem relating the numerical range and the essential numerical range, see \cite{L}. Due to the nonconvexity of the numerical range, we need to introduce the notion of inter-convex hull (see (\ref{intraconvex hull})). The result asserts that the closure of the quaternionic numerical range is precisely the inter-convex hull of the quaternionic essential numerical range and the quaternionic numerical range, \emph{i.e} $\overline{W(T)}=\ich\{W_e(T), W(T)\}$. In spite of the formal similarities with its complex counterpart, there are worth mentioning differences. Foremost we can not infer that the numerical range is closed when it contains the essential numerical range (see remark \ref{remark_lancaster}) as in complex Hilbert spaces \cite[Corollary $1$]{L}. This is because the quaternionic numerical range lacks convexity and the quaternionic Lancaster theorem uses the weaker notion of inter-convex hull. In addition,  remark \ref{remark_lancaster} tells us that, even though the upper bild is convex, we still do not recover Lancaster theorem in its complex form.


\bigskip

\section{Notation and preliminaries}

The division ring of real quaternions $\Hq$, also known as Hamilton quaternions, is an algebra over the field of real numbers with basis $\{1,i,j,k\}$ and product defined by $i^2=j^2=k^2=ijk=-1$. Given a quaternion $q=q_0+q_1i+q_2j+q_3k$, its conjugate is $q^*=q_0-q_1i-q_2j-q_3k$. We call $\re(q)=\frac{q+q^*}{2}$ and $\im(q)=\frac{q-q^*}{2}$ the real and imaginary parts of $q$, respectively. The norm of $q$ is the nonnegative real number $|q|=\sqrt{qq^*}$. Two quaternions $q,q'\in\Hq$ are similar if there is a unitary $u\in\Hq$ such that $u^*qu=q'$, in which case we write $q\sim q'$. This is an equivalence relation and we denote the equivalence class of $q$ by $[q]$.

Let $\mathcal{H}$ denote an infinite dimensional two-sided Hilbert space over $\Hq$. In particular, the norm of $x\in\mathcal{H}$ is defined by the underlying $\Hq$-inner product as $\|x\|=\sqrt{\langle x,x\rangle}$. The inner product verifies the usual Cauchy-Schwartz inequality: $|\langle x,y\rangle|\leq\|x\|\|y\|$, for every $x,y\in\mathcal{H}$. The space of bounded, right $\Hq$-linear operators on $\mathcal{H}$ is denoted by $\mathcal{B}(\mathcal{H})$, its closed ideal of compact operators by $\mathcal{K}(\mathcal{H})$ and the group of invertible operators by $\mathcal{B}(\mathcal{H})^{-1}$.

Every linear operator $T$ considered in the text will be a bounded linear operator in $\mathcal{B}(\mathcal{H})$. Given $q\in\Hq$  and $T \in \mathcal{B}(\mathcal{H})$, we define the operator $\Delta_q(T):\mathcal{H}\to\mathcal{H}$ by
\[
\Delta_q(T)=T^2-2\re(q)T+|q|^2I,
\]
where $I$ is the identity operator. Clearly, $\Delta_q(T)\in\mathcal{B}(\mathcal{H})$. The spherical spectrum of $T$, abbreviated S-spectrum, is the set
\[
\sigma^S(T)=\left\{q\in\Hq:\Delta_q(T)\notin \mathcal{B}(\mathcal{H})^{-1}\right\},
\]
which seems to be the appropriate notion for spectral analysis of linear operators on infinite dimensional quaternionic Hilbert spaces, see \cite{CGSS}.

Let $\pi:\B(\Hi)\rightarrow \B(\Hi)/\K(\Hi)$ denote the canonical quotient map and $\Ca(\Hi)=\B(\Hi)/\K(\Hi)$ the Calkin algebra.
Let $\pi(T)=[T]$ denote the equivalence class $T+\K(\Hi)$, for $T\in \B(\Hi)$. Then $\Ca(\Hi)$ is a normed algebra with $\lVert [T]\rVert=\inf_{K\in\K(\Hi)}\lVert T+K\rVert \leq \lVert T\rVert$. We say that $T$ is a Fredholm operator if the class $[T]$ is invertible in $\Ca(\Hi)$.
According to Atkinson Theorem,
$T\in \B(\Hi)$ is a Fredholm operator if and only if its range is closed and the kernels $\ker(T)$ and $\ker(T^*)$ are finite dimensional, where $T^*\in\mathcal{B}(\mathcal{H})$ is the adjoint of $T$.
The set of all Fredholm operators in $\B(\Hi)$ is denoted by $\Fed(\Hi)$.

The essential S-spectrum of $T\in \B(\Hi)$, defined by
\begin{equation*} \label{eq01}
\sigma_{e}^S(T)=\left\{ q \in \Hq : \Delta_q(T) \notin \Fed(\Hi) \right\},
\end{equation*}
is a non-empty compact subset of $\sigma^S(T)$, see \cite{MT}.


In the sequel, we will be working in the quaternion setting, that is, the quaternions $\Hq$ are our ground field (skewfield to be more precise). Therefore,  when we write $W(T)$ or $W_e(T)$, we always refer to the quaternionic numerical range or quaternionic essential numerical range.

Finally, define the essential bild and the essentials upper and lower bilds to be, respectively, $B_e(T)=W_e(T)\cap\C$, $B_e^+(T)=W_e(T)\cap\C^+$,  and $B_e^-(T)=W_e(T)\cap\C^-$, where $\C^{\pm}$ is the closure of the respective half-planes.


\bigskip

\section{Properties of the essential  numerical range}

This section is devoted to elementary properties of the essential numerical range and to prove some criteria for a quaternion to be in the essential numerical range of an operator. The results and their proofs are identical to the complex case with some adjustments.  For the sake of completeness full proofs are provided. We start with an auxiliary result concerning compact operators.

\begin{lemma}\label{lem compact}
An operator $T$ is compact if and only if $\langle Te_n,e_n\rangle\to 0$ for every orthonormal set $(e_n)_{n}$.
\end{lemma}

\begin{proof}
Let $T\in\mathcal{B}(\Hi)$ be compact and let $(e_n)_{n}$ be an orthonormal set.
Let $P_n$ be the projection onto $\Span\{e_1,\dots, e_n\}$. Since $T$ is compact, it is the limit of a sequence of finite rank operators, i.e.,  $\lim_{n\rightarrow\infty}\lVert P_nT-T\rVert=0$ (see \cite[Corollary 4.5]{C}). Then
\[
\lim_{n\rightarrow\infty} \lVert (I-P_n)T(I-P_n)\rVert \leq\lim_{n\rightarrow\infty}\lVert T-P_nT\rVert \lVert I-P_n\rVert=0.
\]
Since $(I-P_n)e_{n+1}=e_{n+1}$, and using the Cauchy-Schwatz inequality, we have
\begin{eqnarray*}
  \lvert \langle T e_{n+1}, e_{n+1}\rangle\rvert &=& \lvert \langle T(I-P_n) e_{n+1}, (I-P_n)e_{n+1}\rangle \rvert\\
   &\leq &  \lVert (I-P_n)T(I-P_n)\rVert.
\end{eqnarray*}
Hence, $\langle T e_{n}, e_{n}\rangle\to 0$.

For the converse, suppose that $T\in \B(\Hi)$ is such that $\langle T e_{n}, e_{n}\rangle\to 0$, for every orthonormal set $(e_n)_{n}$. From
$\lVert T\rVert=\sup_{\|x\|=\|y\|=1}|\langle Tx, y\rangle|,$
 there exist unit vectors $x_1,y_1\in\Hi$ such that
\begin{equation*}
\lvert \langle Tx_1,y_1\rangle \rvert\geq\frac{\lVert T\rVert }{2}.
\end{equation*}
%

A straightforward computation shows the following  ``polarization identity'', for all $x,y\in\Hi$:
\begin{equation*}
 4\langle Tx,y\rangle = {} \langle T(x+y), x+y \rangle-\langle T(x-y), x-y \rangle + \Big(\langle T(x+yi), x+yi \rangle- \langle T(x-yi), x-yi \rangle\Big)i
 \end{equation*}
\[
+k \Big(\langle T(x+yk), x+yk \rangle- \langle T(x-yk), x-yk \rangle\Big) + k \Big(\langle T(x+yj), x+yj \rangle- \langle T(x-yj), x-yj \rangle\Big)i .
\]
In particular, it follows that
 \[
|\langle Tx_1,y_1\rangle|\leq \frac{1}{4} \sum_{u\in U}
|\langle Tu,u\rangle|,
\]
where $U=\left\{x_1+\eta y_1 \, : \, \eta= \pm 1, \pm i, \pm j, \pm k\right\}$. More precisely, for some $u_0\in U$ we can write
\begin{eqnarray*}
|\langle Tx_1,y_1\rangle| & \leq & \frac{8}{4}
|\langle Tu_0,u_0 \rangle| = 2
 \Bigg| \langle  T\left(\frac{u_0}{\|u_0\|}\right),  \frac{u_0}{\|u_0\|} \rangle\Bigg| \, \|u_0\|^2 \\
& \leq  & 8
 \Bigg| \langle  T\left(\frac{u_0}{\|u_0\|}\right),  \frac{u_0}{\|u_0\|} \rangle\Bigg| ,
   \end{eqnarray*}
where in the last inequality we used the fact that $\|u_0\|\leq 2$.
Set $\rho_1=u_0/ \|u_0\|\in \Hi $. Then, $\rho_1$ is a unit vector such that
\begin{eqnarray*}
\frac{\lVert T\rVert}{2} \leq \lvert \langle Tx_1,y_1\rangle\rvert \leq 8 \lvert \langle T\rho_1, \rho_1\rangle \rvert
\Leftrightarrow
\frac{\lVert T\rVert}{16} \leq  \lvert \langle T\rho_1, \rho_1\rangle \rvert.
 \end{eqnarray*}

Now, let $P_1$ be the orthogonal projection onto $\Span\{\rho_1\}$. By applying the above argument to the operator $(I-P_1)T(I-P_1)$, we can find a unit vector $\rho_2$ orthogonal to $\rho_1$ such that
 \begin{eqnarray*}
   \frac{\lVert (I-P_1)T(I-P_1)\rVert}{16}
\leq   \lvert \langle T\rho_2, \rho_2\rangle \rvert.
 \end{eqnarray*}
Moreover, a recursive procedure allows us to construct an orthonormal sequence $(\rho_n)_{n}$ such that if $P_n$ is the projection onto the span of $\{\rho_1, \dots, \rho_{n}\}$ then
\[
   \frac{ \lVert (I-P_n)T(I-P_n) \rVert}{16} \leq   \lvert \langle T\rho_{n+1}, \rho_{n+1}\rangle  \rvert.
\]
Since $\rho_n$ is an orthonormal sequence, by assumption, we have $\lim_{n\rightarrow \infty}\langle T\rho_n, \rho_n\rangle=0$, so that
 \[
\lim_{n\rightarrow \infty}  \lVert (I-P_n)T(I-P_n)  \rVert=\lim_{n\rightarrow \infty} l\Vert (P_nT+TP_n-P_nTP_n)-T\rVert=0,
 \]
and thus $T$ is compact (being the limit of the finite rank operators $P_nT+TP_n-P_nTP_n$).
\end{proof}


\bigskip

Next result, well-known in the complex setting (see \cite[Corollary in page 189]{FSW}), gives necessary and sufficient conditions for an element $q\in\Hq$ to belong to $W_e(T)$, for some operator $T\in\mathcal{B}(\mathcal{H})$. A very important class of unitary vectors regarding the essential numerical range, portrayed bellow in condition $b)$, will be called an essential sequence, see definition \ref{def_ess_seq}.
As usual, we write $x_n\rightharpoonup x$ if a sequence $(x_n)_n$ in $\mathcal{H}$ converges to $x\in\mathcal{H}$ in the weak topology.

\begin{theorem}\label{theoEssNRseq}
Let $q\in \Hq$. The following conditions are equivalent:
\begin{enumerate}[a)]
  \item $q\in W_e(T)$.
  \item There exists a sequence of unit vectors $(x_n)_{n}$ in $\Hi$ such that $x_n\rightharpoonup 0$  and $\langle Tx_n,x_n\rangle\to q$.
  \item There exists an orthonormal sequence $(e_n)_{n}$ in $\Hi$ such that $\langle Te_n,e_n\rangle\to q$.
\end{enumerate}
\end{theorem}
\begin{proof}
$b)\Rightarrow a).\,$ Suppose b) holds. To see that $q\in \bigcap_{K\in\K(\Hi)}\overline{W(T+K)}$, we will show that $\langle (T+K)x_n,x_n\rangle\to q$, for every compact operator $K$. At this point we need the following well-known result: if $K$ is compact and $x_n\rightharpoonup x$, then $Kx_n\to Kx$ strongly. In particular, if $x_n\rightharpoonup 0$ then $\|Kx_n\|\to 0$.
It follows that
\[
\langle (T+K)x_n,x_n\rangle=\langle Tx_n,x_n\rangle+\langle Kx_n,x_n\rangle\to q,
\]
since we have $\left|\langle Kx_n,x_n\rangle\right|\leq\|Kx_n\|$, for every $n$.

$c)\Rightarrow b).\,$  The result follows from the fact that $e_n\rightharpoonup 0$ for every orthonormal sequence $(e_n)_n$. 

$a)\Rightarrow c)$.
Since $W_e(T)=[B_e(T)]$, it is enough to prove the result for the essential upper bild.
Let $q\in B^+_e(T)$.
From
${B_{e}^+(T)}\subseteq \overline{B^+(T)}$, there is a sequence of unit vectors $(\xi_n)_n$ in $\Hi$ such that $\langle T\xi_n,\xi_n\rangle\in B^+(T)$ and $\lim_{n\rightarrow\infty}\langle T\xi_n,\xi_n\rangle=q$.
Take $\xi_{M}$, which we call without loss of generality $\xi_1$, such that
\[
|\langle T\xi_1,\xi_1\rangle-q|\leq \frac{1}{2}.
\]

Let $\Ls_1:=\Span\{\xi_1\}$ and write $\Hi=\Ls_1\oplus\Ls_1^\bot$. Denote $P_1:\Hi\rightarrow\Hi$ the orthogonal projection onto $\Ls_1$.
From \cite[Corollary 3.3]{CDM3} we know that  the quaternionic numerical range of an operator, and therefore its upper bild,  always intersects the real line. So, we can take a real number $\mu_1\in B^+\Big((I-P_1)\restr{T}{\Ls_1^\bot}\Big)\cap\bR$.

Let $F_1$ be the finite rank operator such that $ T+F_1 =\mu_1P_1+(I-P_1)T(I-P_1)$. Then, $F_1$ compact and,
since $q\in B_e^+(T)$, it follows that
\[
q\in \overline{B^+(T+F_1)}=\overline{B^+(\mu_1P_1+(I-P_1)T(I-P_1))}.
\]
However, it is clear that
\begin{eqnarray*}
& &B^+\left(\mu_1P_1+(I-P_1)T(I-P_1)\right) =\\
   &=& \left\{\langle (\mu_1P_1+(I-P_1)T(I-P_1)) (x_1+x_2), x_1+x_2 \rangle\, :\, (x_1,x_2)\in\Omega\right\} \cap \C^+ \\
   &=& \left\{\mu_1\lVert x_1\rVert^2+\lVert x_2\rVert^2\langle (I-P_1)T \frac{x_2}{\|x_2\|},\frac{x_2}{\|x_2\|} \rangle\, :\,  (x_1,x_2)\in\Omega\right\}\cap \C^+ ,
\end{eqnarray*}
where $\Omega=\left\{(x_1,x_2):x_1\in \Ls_1, x_2\in\Ls_1^\bot, \, \lVert x_1\rVert^2+\lVert x_2\rVert^2=1\right\}$.

 Since $\mu_1\in B^+\Big((I-P_1)\restr{T}{\Ls_1^\bot}\Big)\cap \bR$ and $B^+\Big((I-P_1)\restr{T}{\Ls_1^\bot}\Big)$ is convex (see \cite[Corollary 1]{Ye1}) we obtain
 \[
 B^+(\mu_1P_1+(I-P_1)T(I-P_1)) = B^+\Big((I-P_1)T_{|\Ls_1^\bot}\Big).
 \]
Hence, $q\in \overline{B^+\Big((I-P_1)\restr{T}{\Ls_1^\bot}\Big)}$. So there is a unit vector $\xi_2\in \Ls_1^\bot$  such that
\begin{eqnarray*}
 \lvert \langle (I-P_1)\restr{T}{\Ls_1^\bot}\xi_2,\xi_2\rangle - q\rvert  \leq  \frac{1}{2^2} & \Leftrightarrow & |\langle T \xi_2,\xi_2\rangle - q| \leq  \frac{1}{2^2}.
\end{eqnarray*}

If $\xi_1, \dots, \xi_n$ are orthonormal vectors such that
$|\langle T \xi_n,\xi_n\rangle - q| \leq \frac{1}{2^n},$
we can repeat the above procedure with $\Ls_n:=\Span\{\xi_1, \dots, \xi_n\}$, $P_n$ the orthogonal projection onto $\Ls_n$, $\mu_n\in B^+\Big((I-P_n)\restr{T}{\Ls_n^\bot}\Big)\cap\bR$ and $F_n$ such that $T+F_n=\mu_nP_n+(I-P_n)T(I-P_n)$. We thus obtain a unit vector $\xi_{n+1}$ orthogonal to each $\xi_k$ for $1\leq k\leq n$ such that
\[
|\langle T \xi_{n+1},\xi_{n+1}\rangle - q| \leq \frac{1}{2^{n+1}}.
\]
By recursion, there exists an orthonormal sequence $(\xi_n)_{n}$ in $\Hi$ such that $\langle T\xi_n,\xi_n\rangle\to q$.


\end{proof}

We will call any sequence satisfying b) an essential sequence for $q$, as stated in the following definition.

\begin{definition}\label{def_ess_seq}
An essential sequence $(x_n)_{n}\subset \Hi$ for $q$ is a sequence of unit vectors such that $x_n\rightharpoonup 0$  and $\langle Tx_n,x_n\rangle\to q$. 
\end{definition}

\bigskip
An immediate consequence of theorem \ref{theoEssNRseq} is the non-emptiness of the essential numerical range. In fact, for any orthonormal sequence $(e_n)_{n}$, the sequence $\Big(\langle Te_n,e_n\rangle\Big)_{n}$ is bounded by $\|T\|$. Then, it  has a convergent subsequence.
By c) in theorem \ref{theoEssNRseq} we have that $W_e(T)$ is non-empty. Moreover, it is clear that $W_e(T)$ is a compact set since it is closed and bounded in $\mathbb{H}$. These properties are summarized in the corollary below.

\medskip

\begin{corollary}
$W_e(T)$ is a non-empty and compact set.
\end{corollary}

The essential numerical range in the quaternionic setting shares many properties with either the complex essential numerical range or the quaternionic numerical range. We collect some of such properties below. The proofs are direct and  for that reason only a short hint is provided.
\bigskip

\begin{proposition}\label{prop properties ess n.r}
The following properties of the quaternionic essential numerical range hold.
\begin{enumerate}[(i)]
  \item $W_e(T+K)=W_e(T)$, for all $K\in\K(\Hi)$.
  \item $q\in W_e(T)$ if and only if $[q]\subseteq W_e(T)$.
  \item $W_e(T^*)=W_e(T)$.
  \item $W_e(T)\subseteq \overline{\D(0, \|\pi(T)\|)}$.
  \item If $a,b\in\bR, W_e(aT+bI)=aW_e(T)+b$.
  \item $W_e(T+S)\subseteq W_e(T)+W_e(S)$.
  \item If $U\in\B(\Hi)$ is unitary, then $W_e(UTU^*)=W_e(T)$.
  \item $W_e(T)$ contains  all eigenvalues of $T$ of infinite multiplicity.
\end{enumerate}
\end{proposition}
\begin{proof}
$(i)$ follows from $K+\K(\Hi)=\K(\Hi) $, for any $K \in\K(\Hi)$; $(ii)$  results from  $q\in {W(T)}$ if and only if $[q]\subseteq {W(T)}$, for every operator $T$; $(iii)$ is a consequence of $W(T^*)=W(T)$, for every $T \in \B(\Hi)$; the inclusion $W(T)\subseteq \overline{\D(0, \|T\|)}$ implies $(iv)$; $(v)$ holds because $W(aT+bI)=aW(T)+b$, for $a,b\in\bR$; from $\K(\Hi)+\K(\Hi)=\K(\Hi)$ and $W(T+S)\subseteq W(T)+W(S)$ we obtain $(vi)$; $(vii)$ follows from $W(UTU^*)=W(T)$; for $(viii)$ note that the orthonormal set $(e_n)_{n}$ of eigenvectors satisfying $Te_n=e_n q$ is an essential sequence for $q$.
\end{proof}

\medskip

From \cite[Theorem 2.9]{CDM6} we know that $\sigma^S(T+K) \subseteq \overline{W(T+K)}$, for every $K\in \K(\Hi)$. Using the notion of Weyl S-spectrum,  $\sigma_w^S(T):= \bigcap_{K\in \K(\Hi)}\sigma^S(T+K) $, and that $\sigma_{e}^S(T)\subseteq \sigma_w^S(T)\subseteq \sigma^S(T)$ (see Definition 6.1 and Theorem 6.6 in \cite{MT}), we have the following result.

%
%
%
%
%
%

\begin{theorem}\label{Theo_Sspecturm_NR}
$\sigma_e^S(T)\subseteq W_e(T)$.
\end{theorem}

\section{Convexity}

\bigskip


In this section we establish the main result of the paper which asserts that  the quaternionic essential numerical range is convex. To see this we will show that for any
two elements $\w1, \w2$ in $W_e(T)$, their convex combination can be arbitrarily approximated by elements $\langle Tz, z\rangle$, where $z \in \Hi$ is generated by an essential sequence for $\w1$ and an essential sequence for $\w2$. To construct such elements $z\in\Hi$ we need a preparatory lemma which states a general property enjoyed by a pair of unitary sequences weakly vanishing and a bounded linear operator.


\begin{lemma}\label{lemmasequence}
Let $T\in B(\Hi)$ and $\big(x^{(i)}_n\big)_n$, $i=1,2$, be unitary sequences in $\Hi$ such that $x_n^{(i)}\rightharpoonup 0$. For any $\e>0$ and $N \in \N$, there is $M \in \N$ such that $M\geq N $ and
\[
\big|\inn{x^{(1)}_N}{x^{(2)}_M}\big|\leq \e, \quad \big|\inn{Tx^{(1)}_N}{x^{(2)}_M}\big|\leq \e, \quad\text{and} \quad \big|\inn{T^*x^{(1)}_N}{x^{(2)}_M}\big|\leq \e.
\]
\end{lemma}

\begin{proof}
%
Let $\delta>0$. Let $(e_k)_k$ be an orthonormal basis for $\Hi$ and $P_K$ be the projection onto $\Span\{e_1,\dots, e_K\}$. Since $(I-P_K)y\xrightarrow[K \to \infty]{}0$ for every $y \in \Hi$, then, for the above $\delta>0$ and $N\in\mathbb{N}$, we may find $K \in \N$ such that
\begin{equation}\label{Kcondts}
\|(I-P_K)x^{(1)}_N\|\leq \delta, \quad \|(I-P_K)Tx^{(1)}_N\|\leq \delta \quad\text{ and }\quad \|(I-P_K)T^* x^{(1)}_N\|\leq \delta.
\end{equation}
 We can find an $M\in\mathbb{N}$  that depends on $\delta$, $N$, $K$, such that $M \geq N$ and
\begin{equation}\label{Mcondts}
|\langle x^{(2)}_M, e_k\rangle | \leq \frac{\delta}{2^{k/2}}, \text{ for every } 1 \leq k \leq K.
\end{equation}
Inequality (\ref{Mcondts}) follows  from the fact that $\big(x^{(2)}_n\big)_n$ vanishes weakly, and that implies coordinatewise convergence to zero.
%
It follows that
\[
\Big\|\sum_{1\leq k \leq K} \langle x^{(2)}_M, e_k\rangle \;e_k\Big\|^2 = \sum_{1\leq k \leq K} \Big| \langle x^{(2)}_M, e_k\rangle \Big|^2 \leq \delta^2
\]
and therefore,
\begin{equation} \label{convlemma4}
\big\|P_K x^{(2)}_M\big\|\leq \delta.
\end{equation}

Noting that $\Vert x^{(1)}_N\rVert=\lVert x^{(2)}_M\rVert=1$, we have
\begin{align*}
\Big|\langle x^{(1)}_N, x^{(2)}_M\rangle\Big|  &\leq \Big| \langle x^{(1)}_N, (I-P_K)x^{(2)}_M\rangle \Big| + \Big| \langle x^{(1)}_N, P_K x^{(2)}_M\rangle \Big| \\
&\leq\big\| (I-P_K)x^{(1)}_N\big\| \,\,\big\| x^{(2)}_M \big\| + \big\| x^{(1)}_N\big\|\,\,\big\|P_K x^{(2)}_M\big\|\\
&\leq 2\delta \,\,\,\,\,\,\,(\textrm{from}\,\,(\ref{Kcondts})\,\,\text{and}\,\,(\ref{convlemma4})).
\end{align*}

Using a similar reasoning, we can show that
\begin{align*}
    \big|\langle Tx^{(1)}_N, x^{(2)}_M\rangle\big|&\leq \| (I-P_K)Tx^{(1)}_N\|\,\, \|x^{(2)}_M\|+\|Tx^{(1)}_N\| \| P_K x^{(2)}_M\|\\
&\leq \| (I-P_K)Tx^{(1)}_N\|+\|T\| \,\, \delta \\
&\leq \delta+\|T\| \,\, \delta
\end{align*}
and
$\big|\langle T^*x^{(1)}_N, x^{(2)}_M\rangle\big|\leq \delta+\|T\| \,\, \delta$. Letting $\delta$ be such that $\max\{2, 1+\|T\|\}\delta\leq \e$ the lemma follows.
\end{proof}

\begin{theorem}\label{thm convexity}
 $ W_e(T)$ is convex.
\end{theorem}
\begin{proof}

Convexity of $W_e(T)$ will be proved by showing that $\alpha^2\w1+\beta^2\w2 \in W_e(T)$ for any $\alpha^2+\beta^2=1$, when $\w1, \w2$ lie in $W_e(T)$. For that we will prove there is an essential sequence $(\tilde z_p)_p$ for $\alpha^2\w1+\beta^2\w2$.

 Let $\big(x^{(i)}_n\big)_n$ be an essential sequence for $\omega^{(i)}$ and denote $\omega_n^{(i)}=\langle Tx^{(i)}_n, x^{(i)}_n\rangle$, for $i=1,2$. For any $p\in \N$ let $\e=1/p$. One of the conditions for the sequence $\big(x^{(i)}_n\big)_n$ to be essential for $\w{i}$ is that $\w{i}_n \to \w{i}$ when $n \to \infty$. Hence, for the given $\e$, there exists  $N\geq p$ satisfying
\begin{equation}\label{Ncondts}
|\w1_n-\w1|\leq \e \text{ and  } |\w2_n-\w2|\leq \e, \; \text{ for } n \geq N.
\end{equation}

Pick $M$  according to the previous  lemma. For the fixed $\alpha$ and $\beta$, let $z=\alpha x^{(1)}_N+\beta x^{(2)}_M$. Since $\alpha^2+\beta^2=1$ and $\alpha \beta \leq \frac{1}{2}$, we  easily  verify that
\begin{equation}\label{convlemma1}\\
\big|\|z\|^2-1\big|\leq \big|\langle x^{(1)}_N, x^{(2)}_M\rangle \big|\leq \e.
\end{equation}
A simple computation shows that
 \begin{align*}
 \Big|\langle T z, z\rangle - \big(\alpha^2\w1_N+\beta^2\w2_M\big)\Big| &= \alpha\beta\Big|\langle Tx^{(1)}_N, x^{(2)}_M\rangle + \overline{\langle T^*x^{(1)}_N, x^{(2)}_M\rangle} \Big| \leq \e.\nonumber
 \end{align*}
From (\ref{Ncondts}), it follows that
 \begin{align}
 \Big|\langle T z, z\rangle - \big(\alpha^2\w1+\beta^2\w2\big)\Big|&\leq \Big|  \big(\alpha^2\w1_N+\beta^2\w2_M\big)-\big(\alpha^2\w1+\beta^2\w2\big) \Big|\nonumber\\
 & \,\,\,\,\,\,\,\,\,\,+ \Big|\langle T z, z\rangle - \big(\alpha^2\w1_N+\beta^2\w2_M\big)\Big|\nonumber\\
&\leq \alpha^2 \Big|\w1_N - \w1\big|+ \beta^2 \Big|\w2_M - \w2\big| +\Big|\langle T z, z\rangle - \big(\alpha^2\w1_N+\beta^2\w2_M\big)\Big| \nonumber\\
&\leq  2\e.\label{convlemma2}
  \end{align}

Observing that the fixed integers $N$ and $M$ depend on $\e$, that is on $p\in\N$, we denote them by $N_p$ and $M_p$; likewise, we denote $z$ by $z_p$.  To get an essential sequence we have to normalize $(z_p)_p$. Write $\tilde z_p=\frac{z_p}{\|z_p\|}$. From (\ref{convlemma1}), $\|z_p\|\to 1\,\,(p\to\infty)$, and so $(\tilde z_p)_p$ is well defined.
By definition, $z_p=\alpha x^{(1)}_{N_p}+\beta x^{(2)}_{M_p}$, and $x^{(1)}_{N_p}, x^{(2)}_{M_p} \rightharpoonup 0$, when $p \to \infty$. By linearity and since $\|z_p\|\to 1$, we have that $\tilde z_p \rightharpoonup 0$. Finally, from (\ref{convlemma2}) it follows that
\[
 \langle T \tilde z_p, \tilde z_p\rangle = \frac{1}{\|z_p\|^2} \langle T z_p, z_p\rangle \to \alpha^2\w1+\beta^2\w2.
 \]
The sequence $(\tilde z_p)_p$ is essential for $\alpha^2\w1+\beta^2\w2$ and thus, by theorem \ref{theoEssNRseq}, $\alpha^2\w1+\beta^2\w2\in W_e(T).$
\end{proof}

\bigskip

Next result establishes the relation between the boundary of the numerical range and the essential numerical range. This is the quaternionic analogue of Lancaster's theorem for the complex numerical range, see \cite[Theorem 1]{L}. Since the quaternionic numerical range is not always convex, a modification is imposed and we need to introduce the notion of inter-convex hull of sets (see \cite[Definition 3.2]{CDM2}).

 The inter-convex hull of the sets $A$ and $B$, denoted by $\ich\{A, B\}$, closes the set $A \cup B$ to the convex combinations with one element of each sets,
\begin{equation}\label{intraconvex hull}
 \ich\{A, B\}=\{\alpha a+(1-\alpha) b: a \in A, \, b \in B, \, 0\leq \alpha \leq 1\}.
\end{equation}

\begin{theorem}\label{thm lancaster quaternion}

The closure of the numerical range is $\overline{W(T)}=\ich\{W_e(T), W(T)\}$.
\end{theorem}
\begin{proof}
We start proving that $\ich\{W_e(T), W(T)\} \subseteq\overline{W(T)}$. Let $\bar \omega \in \ich\{W_e(T), W(T)\}$. Then $\bar{\omega}= \alpha^2 \omega+\beta^2 \omega_e$ with $\omega \in W(T), \omega_e \in W_e(T)$ and $\alpha^2+\beta^2=1$.  In particular, we can take a unitary $y \in \Hi$ such that $\omega=\inn{Ty}{y}$ and an essential sequence $(y_n)_n$ for $\omega_e$. Since $y_n \rightharpoonup 0$, we have that $\lim \,\inn{y_n}{y}=\lim \, \inn{y_n}{Ty}=\lim\, \inn{y_n}{T^*y}=0$. Let $z_n=\alpha y+\beta y_n$. Then,
\[
\innt{z_n}=\alpha^2 \innt{y}+\beta^2\innt{y_n} +\alpha\beta\Big(\inn{Ty}{y_n}+\inn{Ty_n}{y}\Big) \to \alpha^2 w+\beta^2w_e=\bar{\omega}.
\]
Furthermore,
\[\|z_n\|^2=\alpha^2\|y\|^2 +\beta^2\|y_n\|^2 +\alpha\beta\Big(\inn{y}{y_n}+\inn{y_n}{y}\Big)\to 1.
\]
Thus $W(T) \ni \innt{\frac{z_n}{\|z_n\|}} \to \bar\omega$, and $\bar\omega \in \overline{W(T)}$.

To prove the converse inclusion, take $\overline\omega \in \overline{W(T)}$. There is a sequence $\left(y_n\right)_n$ in $\Hi$ satisfying $\|y_n\|=1$ and $\omega_n=\langle T y_n, y_n\rangle \to \overline\omega $. Since this sequence is in the unit circle, there is an element $y \in \Hi$ in the unit disk such that $y_n$ converges weakly to $y$.

If $y=0$, then $(y_n)_n$ is an essential sequence for $\overline\omega$. From theorem \ref{theoEssNRseq} we have $\overline\omega \in W_e(T)$.

If $\|y\|=1$, we have that $y_n\rightharpoonup y$, with $\|y\|=1=\|y_n\|$. It is well-known that in this case $y_n \rightarrow y$ (strongly). Thus, $\innt{y_n} \to \innt{y}$, that is, $ \overline\omega =  \innt{y} \in W(T)$.

Assume now that $\|y\| \neq 0,1$. Using that $\langle y_n, h\rangle \to \langle y, h\rangle$ for any $ h \in \Hi$, we can prove that $\lim \;\langle Ty_n, y\rangle=\lim \;\langle Ty, y_n\rangle= \langle Ty, y\rangle$ and therefore
\[
\lim \; \langle Ty_n, y_n\rangle =\lim \;\left[ \langle Ty, y\rangle + \langle T(y_n-y), y_n-y\rangle\right].
\]
It is easy to see that $\lim \|y_n-y\|^2= 1-\|y\|^2$. Then
\begin{align*}
\overline\omega=\lim\;\langle Ty_n, y_n\rangle= & \lim \left[\|y\|^2 \big\langle T\frac{y}{\|y\|}, \frac{y}{\|y\|}\big\rangle + \|y_n-y\|^2\Big\langle  T\frac{y_n-y}{\|y_n-y\|}, \frac{y_n-y}{\|y_n-y\|}\Big\rangle\right]\\
=& \|y\|^2 \big\langle T\frac{y}{\|y\|}, \frac{y}{\|y\|}\big\rangle + (1-\|y\|^2)\lim \Big\langle  T\frac{y_n-y}{\|y_n-y\|}, \frac{y_n-y}{\|y_n-y\|}\Big\rangle .
\end{align*}
We have just written $\overline\omega$ as a convex combination of $\omega=\langle T\frac{y}{\|y\|}, \frac{y}{\|y\|}\big\rangle \in W(T)$ and $\omega_e= \lim \;\langle  T\frac{y_n-y}{\|y_n-y\|}, \frac{y_n-y}{\|y_n-y\|}\Big\rangle$. We use theorem \ref{theoEssNRseq} again, observing that $\left(\frac{y_n-y}{\|y_n-y\|}\right)_n$ is an essential sequence for $w_e$, to conclude that $w_e\in W_e(T)$.  Therefore, $\overline\omega \in \ich\{W_e(T), W(T)\}$.

\end{proof}

As in other results concerning the quaternionic numerical range, next corollary shows that we can simply consider what happens in the complex plane. Given a quaternion $q=q_0+q_1i+q_2j+q_3k$, we define $\pi_{(1)}(q)=\re(q)=q_0$ and $\pi_{(i)}(q)=q_1$.

\begin{corollary}
Let $T\in\B(\Hi)$. Then $\overline{B(T)}=\ich\{B_e(T), B(T)\}$.
\end{corollary}
\begin{proof}
From theorem \ref{thm lancaster quaternion} we have
\[
\ich\{B_e(T), B(T)\}\cap \C \subseteq \ich\{W_e(T), W(T)\}\cap \C=\overline{W(T)}\cap \C.
\]
We obtain that $\ich\{B_e(T), B(T)\} \subseteq \overline{B(T)}$.

For the converse inclusion, take an element $\bar\omega \in \overline{B(T)}$. According to theorem \ref{thm lancaster quaternion} there are $\omega \in W(T)$, $\omega_e \in W_e(T)$ and $\alpha \in [0,1]$, such that
\begin{equation}\label{omega_bar}
\bar\omega =\alpha \omega +(1-\alpha)\omega_e.
\end{equation}
Observe that when $\alpha=0$ or $\alpha=1$ the inclusion immediately follows. So suppose $\alpha\neq 0, 1$.

 We can write $\omega=a+bq$ and $w_e=c+dq_e$, where $a, b, c, d\in \bR$, $q\in \im(q)$,  $q_e\in \im(q_e)$ and $|q|=|q_e|=1$. Therefore, we have
 \[
 \bar\omega = (\alpha a+(1-\alpha)c)+(\alpha bq+(1-\alpha)dq_e).
 \]
 Note that $\alpha bq+(1-\alpha)dq_e\in \Span\{i\}$, since $\bar\omega\in \C$.
 Assume that $\bar\omega\in \C^+$. If  $\bar\omega\in \C^-$, the proof is analogous.

 By circularity of the bild, there are $\omega_{(i)} \in [\omega] \cap \C^+$ in the bild and $\omega_{e,(i)} \in [\omega_e] \cap \C^+$. We can write $\omega_i=a+|b|i$ and $\omega_{e,i}=c+|d|i$.

 Define $\overline{\omega}_i=\alpha\omega_i+(1-\alpha)\omega_{e,i}$, which can be written as
 \[
 \overline{\omega}_i=(\alpha a+(1-\alpha)c)+(\alpha |b|+(1-\alpha)|d|)i.
 \]
Clearly, $\pi_{(1)}(\bar\omega)=\pi_{(1)}(\bar\omega_i)$. On the other hand, since $\alpha bq+(1-\alpha)dq_e\in \Span\{i\}$ and $|q|=|q_e|=1$, we have
\begin{align*}
0 \leq \pi_{(i)}(\bar\omega) &= \Big|\pi_{(i)}(\bar\omega)i\Big|\\
&=\Big|\alpha bq+(1-\alpha)dq_e\Big|\\
&\leq\alpha \lvert b\rvert+(1-\alpha)\lvert d\rvert\\
&=\pi_{(i)}(\bar\omega_i).
\end{align*}

Assuming that $\alpha |b|-(1-\alpha)|d|\geq 0$, let now $\widetilde{\omega}_i=\alpha\omega_i+(1-\alpha)\omega^*_{e,i}$; otherwise, define $\widetilde{\omega}_i=\alpha\omega_i^*+(1-\alpha)\omega_{e,i}$.
Clearly, $\widetilde{\omega}_i\in \C^+$ and $\pi_{(1)}(\widetilde{\omega_i})=\pi_{(1)}(\overline{\omega})$. We have
\begin{align*}
0 \leq \pi_{(i)}(\widetilde{\omega}_i) &= \Big|\alpha |b|-(1-\alpha)|d|\Big|\\
&=\Big|\alpha \lvert bq\rvert-(1-\alpha)\lvert dq_e^*\rvert\Big|\\
&\leq \Big|\alpha bq-(1-\alpha)dq_e^*\Big|\\
&= \Big|\alpha bq+(1-\alpha)dq_e\Big|\\
&= \pi_{(i)}(\bar\omega),
\end{align*}
since $\alpha bq+(1-\alpha)dq_e\in \Span\{i\}$ and $\overline{\omega}\in\C^+$.

Then we have found two elements $\bar \omega_{i}$ and $\tilde \omega_{i}$, both in $\ich\{B_e(T), B(T)\}$, such that
\begin{align*}
  &\pi_{(1)}(\bar\omega)=\pi_{(1)} \big(\bar \omega_{i}\big)=\pi_{(1)}\big(\tilde \omega_{i}\big)\\
  0\leq &\pi_{(i)}\big(\tilde \omega_{i} \big)\leq \pi_{(i)}(\bar \omega) \leq \pi_{(i)} \big(\bar \omega_{i} \big).
\end{align*}

Now we will show that $\overline{\omega}$ is also in $\ich\{B_e(T), B(T)\}$.
Consider the affine transformation
\begin{align*}
  & f:B_e(T)\longrightarrow \C\\
&f(z)=\alpha \omega_i+(1-\alpha)z. \,
\end{align*}
Since $B_e(T)$ is convex, $[\omega^*_{e,i}, \omega_{e,i}]\subset B_e(T)$. Affine transformations map lines into lines so we have
\[
f([\omega^*_{e,i}, \omega_{e,i}])= [\widetilde{\omega}_i, \overline{\omega}_i].
\]
Observe that $\widetilde{\omega}_i\neq \overline{\omega}_i$, since $\alpha\neq 1$.

Since $\overline{\omega}\in [\widetilde{\omega}_i, \overline{\omega}_i]$, there exists $\eta\in [\omega^*_{e,i}, \omega_{e,i}]\subset B_e(T)$ such that $f(\eta)=\overline{\omega},$ that is,
$\alpha \omega_i+(1-\alpha)\eta=\overline{\omega}$.
We conclude that $\overline{\omega}\in \ich\{B_e(T), B(T)\}$.
\end{proof}

\begin{remark}\label{remark_lancaster}
In the complex setting \cite[Corollary $1$]{L} proves that the numerical range is closed if and only if the $W_{\C,e}(T)$ is a subset of the $W_\C(T)$. The relation in the previous result induces the idea that the same result might hold for quaternions. However, that is not the case.

Take the operator $T=\diag\{-1+i, 1+i\}\oplus diag\{s_n\}$, where $s_n$ is a sequence that runs over $(-1/2,1/2)i\cap \mathbb{Q}i$. Applying theorem \ref{theoEssNRseq} and theorem \ref{thm convexity}, we have
\[B_e(T)=[-i/2,i/2].\]
From theorem 4.2 in \cite{CDM6}, it follows
\[
\overline{B^+(T)}=\ch\{-1+i,1+i, -1/3, 1/3] \}.
\]
Nevertheless, the upper bild, and therefore the bild, is not closed. For example, the boundary line segment joining $-1/3$ to $-1+i$ does not belong to $B(T)$. Thus we have $B_e(T)=[-i/2,i/2]\subseteq B(T)$ but $B(T)$ is not closed.
\end{remark}

\end{document}